\documentclass[utf8,11pt]{article}
\usepackage[utf8]{inputenc}
\usepackage{fullpage}
\usepackage{authblk}
\usepackage{amsthm,amsmath,amssymb}
\usepackage{tikz}
\usepackage{subcaption}
\usepackage[colorlinks=false]{hyperref}
\author[1]{Eduard Eiben}
\author[2]{William Lochet}
\author[3,2]{Saket Saurabh\footnote{This project has received funding from the European Research Council (ERC) under the European Union’s Horizon 2020 research and innovation programme (grant agreement No  819416).}}
\affil[1]{ {\small Royal Holloway, University of London, United Kingdom}}
\affil[2]{{\small University of Bergen, Norway }}
\affil[3]{ {\small The Institute of Mathematical Sciences, HBNI, India} }
\title{A Polynomial Kernel for Paw-Free Editing} 
\date{}

\theoremstyle{plain}
\newtheorem{theorem}{Theorem}
\newtheorem{lemma}[theorem]{Lemma}

\theoremstyle{definition}

\newcommand{\bigO}[1]{\ensuremath{{\mathcal O}\left(#1\right)}}

\newcommand{\hedit}{\textsc{$H$-free-Edge Editing}}
\newcommand{\pawedit}{\textsc{Paw-free-Edge Editing}}

\theoremstyle{plain}

\newcommand{\NP}{\textsf{NP}}

\newcommand{\coNP}{\textsf{coNP}}
\newcommand{\NPpoly}{\textsf{NP/poly}}
\newtheorem{rrule}{Reduction Rule}

\begin{document}
	
\maketitle
\begin{abstract}\label{Abstract}
	For a fixed graph $H$, the \hedit\ problem asks whether we can modify a given
	graph $G$ by adding or deleting at most $k$ edges such that the resulting graph
	does not contain $H$ as an induced subgraph. The problem is known to be
	NP-complete for all fixed $H$ with at least $3$ vertices and it admits a
    $2^{\bigO{k}}n^{\bigO{1}}$ algorithm. Cai and Cai [Algorithmica (2015) 71:731–757] showed that 
    \hedit\ does not admit a polynomial kernel whenever $H$
	or its complement is a path or a cycle with at least $4$ edges or a
	$3$-connected graph with at least $1$ edge missing. Their results suggest that
	if $H$ is not independent set or a clique, then \hedit\ admits polynomial
	kernels only for few small graphs $H$, unless $\coNP\in\NPpoly$. Therefore, resolving the kernelization of
	\hedit\ for small graphs $H$ plays a crucial role in obtaining a complete
	dichotomy for this problem. 
	In this paper, we positively answer the question of compressibility for one of
	the last two unresolved graphs $H$ on $4$ vertices.
	Namely, we give the first polynomial kernel for \pawedit\ with $\bigO{k^{6}}$
	vertices. 

\end{abstract}
	
	\newpage
\section{Introduction}\label{Section: Introduction}

For a family of graph $\mathcal{G}$, the general $\mathcal{G}$-\textsc{Graph Modification} 
problem ask whether we can modify a graph $G$ into a graph in $\mathcal{G}$ by
performing at most $k$ simple operations. Typical examples of simple operations that are 
well-studied in the literature include vertex deletion, edge deletion, edge addition, or
combination of edge deletion and addition. We call these problems
$\mathcal{G}$-\textsc{Vertex Deletion}, $\mathcal{G}$-\textsc{Edge Deletion},
$\mathcal{G}$-\textsc{Edge Addition}, and $\mathcal{G}$-\textsc{Edge Editing},
respectively. By a classic result by Lewis and Yannakakis~\cite{LewisY80},
$\mathcal{G}$-\textsc{Vertex Deletion} is \NP-complete for all non-trivial hereditary
graph classes. The situation is quite different for the edge modification problems.
Earlier efforts for edge deletion problems~\cite{ElmallahC88,Yannakakis81}, though
having produced fruitful concrete results, shed little light on a systematic answer, and
it was noted that such a generalization is difficult to obtain.

$\mathcal{G}$-\textsc{Graph Modification} problems have been extensively investigated for
graph classes $\mathcal{G}$ that can be characterized by a finite set of forbidden
induced subgraphs. We say that a graph is $\mathcal{H}$-free, if it does not contain any
graph in $\mathcal{H}$ as an induced subgraph. For this special case, the
\textsc{$\mathcal{H}$-free Vertex Deletion} problem is well understood. If $\mathcal{H}$
contains a graph on at least two vertices, then all of these problems are \NP-complete,
but admit $c^kn^{\bigO{1}}$ algorithm~\cite{Cai96}, where $c$ is the size of the largest
graph in $\mathcal{H}$ (the algorithms with running time $f(k)n^{\bigO{1}}$ are called
fixed-parameter tractable (FPT) algorithms~\cite{CyganFKLMPPS15,DowneyFellows13}). On
the other hand, the \NP-hardness proof of Lewis and Yannakakis~\cite{LewisY80} excludes
algorithms with running time $2^{o(k)}n^{\bigO{1}}$ under Exponential Time Hypothesis
(ETH)~\cite{ImpagliazzoP01}. Finally, as observed by Flum and Grohe~\cite{FlumG06} a
simple application of sunflower lemma~\cite{ErdosR60} gives a \emph{kernel} with
$\bigO{k^{c}}$ vertices, where $c$ is again the size of the largest graph in
$\mathcal{H}$. A kernel is a polynomial time preprocessing algorithm which outputs an
equivalent instance of the same problem such that the size of the reduced instance is
bounded by some function $f(k)$ that depends only on $k$. We call the function $f(k)$
the size of the kernel. It is well-known that any problem that admits an FPT algorithm
admits a kernel. Therefore, for problems with FPT algorithms one is interested in
polynomial kernels, i.e., kernels where size upper bounded by a polynomial function.

For the edge modification problems, the situation is more complicated. While all of
these problems also admit $c^kn^{\bigO{1}}$ time algorithm, where $c$ is the maximum
number of edges in a graph in $\mathcal{H}$~\cite{Cai96}, the \textsf{P} vs \NP\
dichotomy is still not known. Only recently Aravind et al.~\cite{AravindSS17b} gave the
dichotomy for the special case when $\mathcal{H}$ contains precisely one graph
$H$~\cite{AravindSS17b}. From the kernelization point of view, the situation is even 
more difficult. The reason is that deleting or adding an edge to a graph can introduce a
new copy of $H$ and this might further propagate. Hence, we cannot use the sunflower
lemma to reduce the size of the instance. Cai asked the question whether
\textsc{$H$-free Edge Deletion} admits a polynomial kernel for all graphs $H$~\cite{bodlaender2006open}. 
Kratsch and Wahlstr{\"{o}}m~\cite{KratschW13} showed
that this is probably not the case and gave a graph $H$ on $7$ vertices such that 
\textsc{$H$-free Edge Deletion} and \textsc{$H$-free Edge Editing} does not admit a
polynomial kernel unless $\coNP\subseteq \NPpoly$. Consequently, it was shown that this 
is not an exception, but rather a rule~\cite{CaiCai15,GuillemotHPP13}. Indeed the result by Cai and 
Cai~\cite{CaiCai15} shows that \textsc{$H$-free Edge Deletion}, \textsc{$H$-free Edge Addition}, 
and \hedit\ do not admit a polynomial kernel whenever $H$ or its complement is a path or a cycle 
with at least $4$ edges or a $3$-connected graph with at least $2$ edges missing. This suggests that 
actually the $H$-free modification problems with a polynomial kernels are rather rare and only for
small graphs $H$. For the graphs on $4$ vertices the kernelization of $H$-free edge
modification problems was open for last two graphs and their complements (see
Table~\ref{table:1}), namely paw and claw, and Cao et al.~\cite{CaoRSY18} conjectured
that all of these problems admit polynomial kernels. In this paper, we give kernels for
the first of the two remaining graphs, namely the paw.

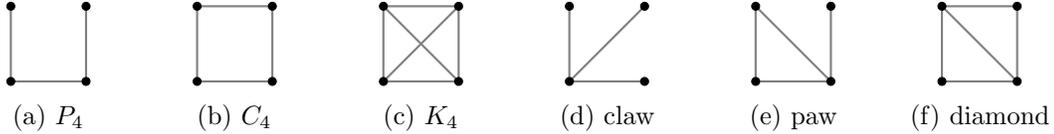
\begin{figure}
	\centering
	\begin{minipage}[b]{.15\linewidth}
		\centering
		\begin{tikzpicture}[scale=0.50]
		\draw[gray, thick] (-1,1) -- (1,1);
		\draw[gray, thick] (-1,1) -- (-1,3);
		\draw[gray, thick] (1,1) -- (1,3);
		\filldraw[black] (-1,1) circle (3pt);
		\filldraw[black] (-1,3) circle (3pt);
		\filldraw[black] (1,1) circle (3pt);
		\filldraw[black] (1,3) circle (3pt);
		\end{tikzpicture}
		\subcaption{$P_4$}\label{fig:1a}
	\end{minipage}%
	\begin{minipage}[b]{.15\linewidth}
	\centering
	\begin{tikzpicture}[scale=0.50]
	\draw[gray, thick] (-1,1) -- (1,1);
	\draw[gray, thick] (-1,1) -- (-1,3);
	\draw[gray, thick] (1,1) -- (1,3);
	\draw[gray, thick] (1,3) -- (-1,3);
	\filldraw[black] (-1,1) circle (3pt);
	\filldraw[black] (-1,3) circle (3pt);
	\filldraw[black] (1,1) circle (3pt);
	\filldraw[black] (1,3) circle (3pt);
	\end{tikzpicture}
	\subcaption{$C_4$}\label{fig:1a}
\end{minipage}%
	\begin{minipage}[b]{.15\linewidth}
	\centering
	\begin{tikzpicture}[scale=0.50]
	\draw[gray, thick] (-1,1) -- (1,1);
	\draw[gray, thick] (-1,1) -- (-1,3);
	\draw[gray, thick] (1,1) -- (1,3);
	\draw[gray, thick] (-1,1) -- (1,3);
	\draw[gray, thick] (1,1) -- (-1,3);
	\draw[gray, thick] (-1,3) -- (1,3);	
	\filldraw[black] (-1,1) circle (3pt);
	\filldraw[black] (-1,3) circle (3pt);
	\filldraw[black] (1,1) circle (3pt);
	\filldraw[black] (1,3) circle (3pt);
	\end{tikzpicture}
	\subcaption{$K_4$}\label{fig:1a}
\end{minipage}%
	\begin{minipage}[b]{.15\linewidth}
	\centering
	\begin{tikzpicture}[scale=0.50]
	\draw[gray, thick] (-1,1) -- (1,1);
	\draw[gray, thick] (-1,1) -- (-1,3);
	\draw[gray, thick] (-1,1) -- (1,3);
	\filldraw[black] (-1,1) circle (3pt);
	\filldraw[black] (-1,3) circle (3pt);
	\filldraw[black] (1,1) circle (3pt);
	\filldraw[black] (1,3) circle (3pt);
	\end{tikzpicture}
	\subcaption{claw}\label{fig:1a}
\end{minipage}%
	\begin{minipage}[b]{.15\linewidth}
	\centering
	\begin{tikzpicture}[scale=0.50]
	\draw[gray, thick] (-1,1) -- (1,1);
	\draw[gray, thick] (-1,1) -- (-1,3);
	\draw[gray, thick] (1,1) -- (1,3);
	\draw[gray, thick] (1,1) -- (-1,3);
	\filldraw[black] (-1,1) circle (3pt);
	\filldraw[black] (-1,3) circle (3pt);
	\filldraw[black] (1,1) circle (3pt);
	\filldraw[black] (1,3) circle (3pt);
	\end{tikzpicture}
	\subcaption{paw}\label{fig:1a}
\end{minipage}%
	\begin{minipage}[b]{.15\linewidth}
	\centering
	\begin{tikzpicture}[scale=0.50]
	\draw[gray, thick] (-1,1) -- (1,1);
	\draw[gray, thick] (-1,1) -- (-1,3);
	\draw[gray, thick] (1,1) -- (1,3);
	\draw[gray, thick] (1,3) -- (-1,3);
	\draw[gray, thick] (1,1) -- (-1,3);
	\filldraw[black] (-1,1) circle (3pt);
	\filldraw[black] (-1,3) circle (3pt);
	\filldraw[black] (1,1) circle (3pt);
	\filldraw[black] (1,3) circle (3pt);
	\end{tikzpicture}
	\subcaption{diamond}\label{fig:1a}
\end{minipage}%
	\caption{Graphs on $4$ vertices, their complements are omitted.}
	\label{Figure:1}
\end{figure}

\begin{table}
	\centering
	\begin{tabular}{l|ccc}
		\hline 
		$H$ & deletion & addition & editing \\
		\hline
		$K_4$ & $\bigO{k^4}$~\cite{Caithesis} & trivial & $\bigO{k^4}$~\cite{Caithesis} \\
		$P_4$ & $\bigO{k^3}$~\cite{GuillemotHPP13} & $\bigO{k^3}$~\cite{GuillemotHPP13} & $\bigO{k^3}$~\cite{GuillemotHPP13} \\
		diamond & $\bigO{k^3}$~\cite{SandeepS15} & trivial & $\bigO{k^8}$~\cite{CaoRSY18} \\ 
		paw & $\bigO{k^3}$~[this paper] & $\bigO{k^3}$~[this paper] & $\bigO{k^6}$~[this paper] \\
		\hline
		claw & open & open & open \\
		\hline 
		$C_4$ & no~\cite{GuillemotHPP13} & no~\cite{GuillemotHPP13} & no~\cite{GuillemotHPP13} 
	\end{tabular}
\caption{The kernelization results of $H$-free edge modification problems for $H$ being $4$-vertex graphs. Note that for a complement of $H$, the rows with deletion and addition are swapped, but otherwise the same results hold.}\label{table:1}
\end{table}

\subsection{Brief Overview of the Algorithm} Our main result is a polynomial kernel for
\pawedit. The key to obtain the kernel is a structural theorem by
Olariu~\cite{OLARIU198853} that states that every connected paw-free graph is either
triangle-free or complete multipartite graph. We start our kernelization algorithm
by finding a greedy edge-disjoint packing of paws in $G$. This clearly contains at most
$k$ paws and hence at most $4k$ vertices. Let us denote the set of these vertices by $S$.
The goal now is to bound the vertices in $G-S$.
Bounding the vertices belonging to the complete multipartite components of $G-S$ is rather 
simple. We show that every
vertex in $S$ is adjacent to at most $1$ complete multipartite component and for each
multipartite component, we can reduce the size of each part as well as the number of
these parts to $\bigO{k}$, else we can always find an irrelevant vertex that does not
appear in any solution. The triangle-free part is more tricky. The difficulty comes from
the fact that actually instead of keeping this part of the graph triangle-free, the
optimal solution might want to add some edges to make it complete multipartite. 
We are
however able to show that there is always optimal solution that keeps the vertices at
distance at least $5$ from $S$ in a triangle-free component. This structural claim helps us in looking 
for solution which are not too far away from $S$ ``in some sense''. 
Moreover, after some preprocessing of the instance, we can also show that the vertices with 
more than $4k+6$ neighbors inside the triangle-free components of $G-S$ cannot end up inside a 
complete multipartite component. It means that we can mark the relevant vertices in 
triangle-free components as follows. Set $S_0 := S$ and for every $i <5$, let $S_{i+1}$ 
be the set obtained by marking for each vertex of $S_{i+1}$, $4k+6$ neighbors at distance 
$i+1$ from $S$. The set of vertices marked is then $\bigO{k^6}$. Finally,
we can remove the vertices of triangle-free components which have not been 
marked. This is safe because these vertices are either too far from $S$ to belong to a complete multipartite component, 
or every way to connect these vertices to $S$ use vertices that can't end up in a complete multipartite
component of the reduce instance because of the degree condition. 
This gives us the desired kernel.


\section{Preliminaries}\label{Section: Preliminaries}
We assume familiarity with the basic notations and terminologies in graph theory. We refer the reader
to the standard book by Diestel~\cite{diestel} for more information. Given a graph $G$ and a set of 
pairs of vertices $A\in V(G)^2$, we denote by $G \Delta A$ the graph whose set of vertices is $V(G)$ 
and set of edges is the symmetric difference of $E(G)$ and $A$.
 
\smallskip
\noindent \emph{Parameterized Algorithms and Kernelization:} For a detailed illustration of the following facts the reader is
referred to~\cite{CyganFKLMPPS15,DowneyFellows13}.
A \emph{parameterized problem} is a language $\Pi \subseteq
\Sigma^*\times \mathbb{N}$, where $\Sigma$ is a finite alphabet; the second
component $k$ of instances $(I,k) \in \Sigma^*\times\mathbb{N}$ is called the
\emph{parameter}. A parameterized problem $\Pi$ is
\emph{fixed-parameter tractable} if it admits a
\emph{fixed-parameter algorithm}, which decides instances $(I,k)$ of
$\Pi$ in time $f(k)\cdot |I|^{\bigO{1}}$ for some computable function
$f$.  
 
A \emph{kernelization} for a parameterized
problem $\Pi$ is a polynomial-time algorithm that given any instance
$(I,k)$ returns an instance $(I',k')$ such that $(I,k) \in \Pi$ if and
only if $(I',k') \in \Pi$ and such that $|I'|+k'\leq f(k)$ for some
computable function $f$. The function $f$ is called the \emph{size}
of the kernelization, and we have a polynomial kernelization if $f(k)$
is polynomially bounded in $k$. It is known that a parameterized
problem is fixed-parameter tractable if and only if it is decidable
and has a kernelization. However, the kernels implied by this fact are
usually of superpolynomial size.

A \emph{reduction rule} is an algorithm that takes as input an
instance $(I,k)$ of {a parameterized problem $\Pi$} and outputs an instance
$(I',k')$ of the same problem. We say that the reduction rule is
\emph{safe} if $(I,k)$ is a \emph{yes}-instance if and only if $(I',k')$ is a \emph{yes}-instance. In order to describe our kernelization
algorithm, we present a series of reduction rules. 

We will need the following result describing the structure of paw-free graphs 
\cite{OLARIU198853}. 

\begin{theorem}\label{th:structure}
    $G$ is a paw-free graph if and only if each connected component of $G$ is triangle-free or
    complete multipartite. 
\end{theorem}

To make a clear distinction between these two cases, we will say that a graph is 
a complete multipartite graph if it contains at least three parts. In particular, it contains a 
triangle. 

\section{Reduction Rules}

From now on $(G,k)$ will be an instance of paw-free editing and we assume $k >0$. 
Let us first describe two rules which can be safely applied. 

\begin{rrule}\label{rrule:twins}
If $X$ is an independent set of $k+3$ vertices with the same neighborhood, remove a vertex $x \in X$ from the graph.
\end{rrule}
\begin{proof}[Proof of Safeness]
	Suppose $(G,k)$ is an instance of the paw-free editing problem and  $X$ is an independent set of 
	$k+3$ vertices with the same neighborhood. Let $G'$ be the graph obtained by removing a vertex of 
    $X$. We need to show that $(G',k)$ has a solution if and only if $(G,k)$ has one. Since $G'$ is a
    subgraph of $G$, it is clear that if $(G,k)$ has a solution, then so does $(G',k)$. 
    Let $A$ be a solution to $(G',k)$ and assume $G \Delta A$ contains a paw $x_1, x_2, x_3, x_4$ with $x_1, x_2, 
	x_3$ being a triangle and $x_4$ being adjacent to $x_3$. Because $A$ is a solution to $(G',k)$, 
	it means that one of the $x_i$ must be the vertex $x$ that we removed from $G$. Moreover, 
    at most two of the other vertices of $X$ belong to the paw, as $x$ is adjacent to at least one vertex
    and $X$ is an independent set. If only one other vertex of $X$ 
	belongs to it, consider the other $k+1$ vertices of $X$ which are not in the paw. They all have 
	the same neighborhood in the paw as $x$, so $A$ must contain for each of them at least one edge 
	with the paw, or we could replace $x$ with this vertex in the paw, which contradicts the fact 
	that $A$ is a solution of $(G',k)$. However, since $A$ is smaller than $k+1$ we reach a 
	contradiction. If two other vertices of $X$ belong to the paw, then it means that $x = x_4$ and 
	these vertices are $x_1$ and $x_3$. Moreover it means that the edge $x_1x_3$ must be edited as 
	$X$ is an independent set. In that case, consider the other $k$ vertices of $X$ which are not in the paw. 
	Again, for each of them, the solution must contains an edge with the paw, but since $|A \setminus
	( x_1 x_3) |< k$, we also reach a contradiction. Overall this implies that Rule 1 is safe. 
\end{proof}

Following analogous arguments for the case when $X$ induces a complete multipartite graph with at 
least $k+5$ parts, we also obtain safeness of the following rule. 
\begin{rrule}\label{rrule:large_complete_multipartite}
If $X$ is a complete multipartite subgraph with $k+5$ parts having the same neighborhood outside of 
$X$, then remove the smallest part of $X$ from the graph.
\end{rrule}

\begin{proof}[Proof of Safeness]
    Suppose $(G,k)$ is an instance of the paw-free editing problem and $X$ is a complete multipartite 
    subgraph with $k+5$ parts having the same neighborhood outside of $X$. Let $G'$ be the graph 
    obtained by removing the smallest part $P$ of $X$. 
    We need to show that $(G',k)$ has a solution if and only if $(G,k)$ has one. Let $A$ be a
	solution to $(G',k)$ and assume $G \Delta A$ contains a paw $x_1, x_2, x_3, x_4$ with $x_1, x_2, 
	x_3$ being a triangle and $x_4$ being adjacent to $x_3$. Because $A$ is a solution to $(G',k)$, 
    it means that one of the $x_i$ must belong to $P$. Moreover, since the vertices in $P$ have 
    exactly the same neighborhood in $G$ and they form an independent set, this paw can contain at most 
    one vertex from $P$. Let us call $x$ this vertex. 
    Since $X$ consists of $k +5$ parts, it means that there exists $k+1$ 
    parts different from $P$ and without a vertex in this paw. However we know that any vertex in 
    these parts has the exact same neighborhood as $x$ inside the paw. This means that each of these 
    vertices must be adjacent in $A$ to the paw, or we can replace $x$ with a vertex belonging to $G'$,
    which is a contradiction. However, since there is at least $k+1$ of these vertices and $|A| = k$, 
    we reach a contradiction.
\end{proof}

Note that if there exists a set $X$ for which Reduction Rule~\ref{rrule:twins} can be applied, then this set can be found in polynomial time. 
Therefore from now on we assume that $(G,k)$ is an instance where 
Reduction Rule~\ref{rrule:twins} cannot be applied. 
Let $\mathcal{H}$ be a maximal packing of edge-disjoint paws and $S$ the set of vertices
appearing in $\mathcal{H}$.

We will now introduce two new rules.

\begin{rrule}\label{rrule:triangles_sharing_edge}
	If there is a pair of adjacent vertices $s_1, s_2$ with $4k+6$ common neighbors
	in the triangle-free components of $G-S$, 
	then remove the edge $s_1,s_2$ and set $k:= k-1$. 
\end{rrule}

The soundness of Reduction Rule~\ref{rrule:triangles_sharing_edge} is implied by the following Lemma:

\begin{lemma}\label{lem:triangle_2}
    Suppose Reduction Rule~\ref{rrule:twins} cannot be applied anymore and 
    let $s_1, s_2$ be two adjacent vertices. If there are more than $4k+6$ vertices belonging to the 
    triangle-free components of $G-S$ adjacent to 
    both $s_1$ and $s_2$, then either $(G,k)$ is a \emph{no}-instance, or any solution uses the edge 
    $s_1s_2$.  
\end{lemma}

\begin{proof}
    Suppose there is a solution $A$ not using the edge $s_1s_2$. Because $s_1$ and $s_2$ have $4k+6$
    common neighbors in $G$, it means that they belong to a triangle and thus
    to a complete multipartite component of $G \Delta A$. Because $|A| = k$, 
    we know that at least $2k+6$ of the common neighbors of $s_1$ and $s_2$ are not adjacent to any 
    edge in $A$. This means that these vertices belong to the same component in $G \Delta A$, 
    and moreover they can only be in two different parts as they belong to the triangle-free components
    of $G-S$. This means that $k+3$ of these vertices belong to the same part of a complete 
    multipartite component of  $G \Delta A$ and since they are not incident to any edge in $A$, they
    have the same neighborhood in $G$. Therefore, we could have applied Reduction Rule~\ref{rrule:twins}. 
\end{proof}

\begin{rrule}\label{rrule:large_part}
	If $C$ is a complete multipartite component of $G-S$ and $P_1$ is a part of 
	$C$ with more than $3k+3$ vertices, then remove all the edges between the other parts of $C$ and 
	decrease $k$ by the amount of edges removed. If this amount is greater than $k$, answer \emph{no}.
\end{rrule} 

The fact that Reduction Rule~\ref{rrule:large_part} is safe is implied by the following Lemma:

\begin{lemma}
    Suppose Reduction Rule~\ref{rrule:twins} cannot be applied anymore and assume $C$ is a complete multipartite component of 
    $G-S$. If one part of $C$ is larger than $3k+3$, then either $(G,k)$ is a \emph{no}-instance, or any
    solution will remove all the edges between the other parts of $C$.
\end{lemma}

\begin{proof}
    Let $P_1$ be a part of $C$ of size greater than $3k+3$ and let $s_1, s_2$ be two adjacent vertices 
    of $C - P_1$. Let $A$ be a solution of size at most $k$ which does not use the edge $s_1s_2$.
    $A$ is incident to at most $2k$ vertices, so it means that at least $k+3$ vertices of $P_1$ are
    not incident to any edge of $A$. Moreover, since $s_1s_2$ is not in $A$, these $k+3$ vertices belong
    to the same part of a complete multipartite component of $G \Delta A$ and thus have the same
    neighborhood in $G$. This is a contradiction, as Reduction Rule~\ref{rrule:twins} cannot be applied anymore. 
\end{proof}

Note also that if Reduction Rules~\ref{rrule:triangles_sharing_edge}~and~\ref{rrule:large_part} 
can be applied, then it is possible to do it in polynomial time.
From now on assume that none of these rules can be applied. 

\section{Bounding the Complete Multipartite Components}
The next two lemmas allow us to bound the number of vertices belonging to complete multipartite 
components of $G-S$.

\begin{lemma}\label{lem:mcc_size}
Let $C$ denote a complete multipartite component of $G - S$. If $|C| \geq (3k+3)(3k+5)$, then either
Reduction Rule~\ref{rrule:large_complete_multipartite} can be applied or $(G,k)$ is a
\emph{no}-instance. Moreover, if Reduction Rule~\ref{rrule:large_complete_multipartite} can be applied, then it can
be done in polynomial time.
\end{lemma}

\begin{proof}
    Because Reduction Rule~\ref{rrule:large_part} cannot be applied, we have that every 
    part of $C$ contains at most $(3k+3)$ vertices. 
    Suppose now that $C$ consists of more than $3k+5$ parts. If $(G,k)$ is a \emph{yes}-instance, 
    then the solution can only be adjacent to at most $2k$ of these parts. The complete multipartite
    graph consisting of the $k+5$ parts not adjacent to the solution is then a candidate to apply 
    Reduction Rule~\ref{rrule:large_complete_multipartite}. 

    Note that to find the multipartite subgraph to apply Reduction Rule~\ref{rrule:large_complete_multipartite}, we only have to check for each part 
    if the vertices in this part have the same neighborhood outside of $C$, and for the part that do, 
    find a maximum set of parts with the same neighborhood.
\end{proof}

\begin{lemma}\label{lem:mcc_number}
For any $s \in S$, $s$ is adjacent to at most one complete multipartite component of $G-S$. 
\end{lemma}

\begin{proof}
    Suppose $s \in S$ is adjacent to two complete multipartite components $C_1$ and $C_2$. 
    Let $x$ be a vertex of $C_1$ adjacent to $s$. By definition of $C_1$, there exist vertices 
    $y$ and $z$ in $C_1$ such that $x,y,z$ is a triangle. This implies that one of $y$ and $z$ 
    has to be adjacent to $s$ or it would yield a paw without any edge in $S$ which is not possible
    by definition of $\mathcal{H}$. 
    
    Suppose now that $y$ is adjacent to $s$ (the case $x$ is adjacent to $s$ is identical). 
    Now let $c_2$ be a vertex of $C_2$ adjacent to $s$. Because 
    $C_1$ and $C_2$ are two different components, $c_2$ cannot be adjacent to either $c_1$ or $y$, 
    which means that $s,c_1, c_2$ and $y$ form  paw without any edge in $S$, a contradiction.
\end{proof}

The next section is devoted to proving that, if there exists a solution $A$, then we can assume
that any complete multipartite component of $G\Delta A$ only contains vertices at distance $5$ from $S$. 

\section{Bounding the Diameter of Relevant Vertices}

Let $A$ denote a solution such that the sizes of the multipartite components in $G \Delta A$ are minimal. 
In this section, $C$ will denote a complete multipartite component of $G\Delta A$, 
and $C_1, C_2, \dots, C_r$ the parts of $C$.
For any $i \in [r]$ and $j$, let $C_{i,j}$ denote the set of vertices of $C_i$ which 
are at distance $j$ of $S$ and $\overline{C_{i,j}} = \bigcup_{t \not = i} C_{t,j}$.

\begin{lemma}\label{lem:cc1}
    For any $j \geq 4$, and any $i \in [r]$, if $C_{i,0} \cup C_{i,1}$ is non empty, then $C_{i,j}$
    is. 
\end{lemma}

\begin{proof}
    Suppose $C_{i,0} \cup C_{i,1}$ and  $C_{i,j}$ are non empty.

    Because $j \geq 4$, we know that $E(C_{i,j}, \overline{C_{i,0}} \cup 
    \overline{C_{i,1}} \cup \overline{C_{i,2}}) $ is empty. This implies that $A$ contains all 
    the pairs in $C_{i,j} \times (\overline{C_{i,0}} \cup \overline{C_{i,1}} \cup 
    \overline{C_{i,2}})$. However, vertices in $C_{i,j}$ can only be adjacent to vertices at distance 
    $i, i-1$ and $i+1$ from $S$, thus replacing all the edges in $C_{i,j} \times (\overline{C_{i,0}}
    \cup \overline{C_{i,1}} \cup \overline{C_{i,2}})$ by the pairs in $E(C_{i,j}, 
    \overline{C_{i,j-1}} \cup \overline{C_{i,j}} \cup \overline{C_{i,j+1}} )$ would also give a
    solution by disconnecting the vertices in $C_{i,j}$ from $C$. However, since $A$ is chosen such 
    that $|C|$ is minimal, it implies that:
    $ |\overline{C_{i,j-1}} \cup \overline{C_{i,j}} \cup \overline{C_{i,j+1}}|  \geq  
    |\overline{C_{i,0}} \cup \overline{C_{i,1}} \cup \overline{C_{i,2}}|$. 

    Now setting $$A' := \left(A \cup E(C_{i,0} \cup C_{i,1}, \overline{C_{i,0}} \cup 
    \overline{C_{i,1}} \cup \overline{C_{i,2}} ) \right)  \setminus \left( (C_{i,0} \cup C_{i,1} ) \times
    (\overline{C_{i,j-1}} \cup \overline{C_{i,j}} \cup \overline{C_{i,j+1}}) \right)  $$ gives an optimal 
    solution where $C$ doesn't contain $C_{i,0} \cup C_{i,1}$ and whose value is as good as $A$, 
    a contradiction.
\end{proof}
\noindent
For any $j$, let $S_j = \bigcup_{i \in [r]} C_{i,j} $. In other word, $S_j$ is the set of vertices
of $C$ at distance $j$ from $S$. The main implication of Lemma \ref{lem:cc1} is that, if $S_j$ is 
not empty for $j \geq 4$, then $A$ contains all the pair $S_i \times (S_0 \cup S_1)$. Indeed, 
it shows that vertices in $S_j$ and $S_0 \cup S_1$ belongs to different parts and thus must be adjacent
in $G \Delta A$. However, just by considering the distance to $S$ in $G$, these vertices cannot be
adjacent in $G$, and thus these pairs must be in $A$. This allows us to prove the following lemma. 

\begin{lemma}\label{lem:distance}
    For any $j \geq 5$, $S_j$ is empty.
\end{lemma}

\begin{proof}
    Suppose $S_4$ and $S_5$ are non empty. By Lemma \ref{lem:cc1}, we know that the vertices in 
    $S_5$ and $S_0 \cup S_1$ belong to different parts of the complete multipartite component. 
    This implies that $A$ contains 
    $S_5 \times (S_0 \cup S_1)$. However, removing these pairs from $A$, as well as all pairs containing 
    a vertex of $C$ at distance more than 6 from $S$, and adding $E_G(S_5, S_4)$ also 
    yields a solution by disconnecting $S_5$ from the multipartite component. 
    By optimality of $A$, this implies that $E_G(S_5, S_4) \geq |S_5| |S_0 \cup 
    S_1|$ and thus $|S_4| \geq |S_1 \cup S_0|$. 
    Now again by Lemma \ref{lem:cc1}, we have that $A$ contains $S_4 \times (S_0 \cup S_1)$. 
    However, $|S_4| \geq |S_1 \cup S_0|$ so it means that 
    $|S_1 \cup S_0|^2 \leq |S_4| |S_1 \cup S_0|$. Let $A'$ be the solution obtained from $A$ 
    by disconnecting $S_1$ from $S_0$ and removing all pairs adjacent to the sets $S_j$ for 
    $j \geq 2$. Because  $|S_1 \cup S_0|^2 \leq |S_4| |S_1 \cup S_0|$, we have that $|A'| \leq |A|$
    and the multipartite component containing $S_0$ is strictly smaller in $G \Delta A'$ than 
    in $G \Delta A$ while the other remain exactly the same, which is a contradiction. 
\end{proof}

\section{Triangle-Free Components}

Before proving our main result let us prove the following lemma, which will be useful in bounding the
number of vertices outside of $S$.

\begin{lemma}\label{lem:degre_trianglefree}
    If $x \in G$ has at least $4k + 6$ neighbors belonging to triangle-free 
    components of $G-S$, then there is no solution $A$ such that $x$ belongs to a complete 
    multipartite component of $G \Delta A$. 
\end{lemma}

\begin{proof}
    Let $T$ denote the set of neighbors of $x$ belonging to triangle-free components of $G-S$.
    Suppose $x$ belongs to a complete multipartite component $C$ of $G \Delta A$. First note that 
    at least $2k+6$ of the vertices of $T$ will not be adjacent to any edge of $A$, which means that 
    their neighborhood in $G$ and $G \Delta A$ are the same and they belong to $C$ in $G \Delta A$. 
    Now because the vertices of $T$ belong to triangle-free components, it means that these $2k+6$ 
    vertices can only belong to two different
    parts of this multipartite component. In particular, at least $k+3$ of those belong to the same 
    part and thus have the exact same neighborhood in $G \Delta A$ and thus in $G$. This means that 
    Reduction Rule~\ref{rrule:twins} can be applied, which is a contradiction.  
\end{proof}

\begin{lemma}\label{lemma:triangle_1}
Suppose $(G,k)$ is a \emph{yes}-instance. Then there exists a set $S'$ of at most 
$( 4k+6)4k$ vertices such that if $x \not 
\in S'$ belongs to a triangle-free component of $G-S$, then $x$ doesn't belong to any triangle in 
$G$ using only one vertex of $S$. Moreover, there is a polynomial
time algorithm that either find this set or concludes that $(G,k)$ is a \emph{no}-instance.  
\end{lemma}

\begin{proof}
    Let $x$ be a vertex belonging to a triangle-free component $C$ of $G-S$.
    Suppose that $x$ belongs to a triangle using only one vertex $s$ of $S$ and another vertex
    $y$ of $C$. Note first that $C$ is the only component of $G-S$ adjacent to $s$ or we would have a paw 
    using edges not in $S$. Suppose now that $t \in C$ is adjacent to $x$. Then $t$ must be adjacent 
    to either $y$ or $s$ or it would yield a paw using no edge in $S$. Thus, since $C$ is triangle 
    free, $t$ must be adjacent to $s$. The same argument would show that any vertex adjacent to $t$ 
    in $C$ must be adjacent to $s$ and thus the whole component $C$ is adjacent to $x$. 

    Let $\mathcal{M}$ be a maximal matching in $C$. If $\mathcal{M}$ consists of more than $k$ edges, 
    then it means that any solution $A$ to the instance $(G,k)$ puts $s$ in a complete multipartite 
    component. In particular if $|C| \geq 4k+6$, as $C \subseteq N(x)$ and $|A| \geq k$,
    we have that $2k+6$ of the vertices of $C$ are not adjacent to any edge of $A$ and belong to 
    the same complete multipartite component as $s$. Moreover, these vertices can only belong to 
    two different parts of this complete multipartite component (or we would have a triangle in $C$),
    and thus $k+3$ of them belong to the same part. However, since their neighborhood in $G$ 
    and $G \Delta A$ are identical, it means we could have applied Reduction Rule~\ref{rrule:twins}, so $(G,k)$ is a 
    \emph{no}-instance. So let $C'$ be defined as the vertices of $\mathcal{M}$ if $|\mathcal{M}| 
    \leq k$ and the full set $C$ if $\mathcal{M}$ if $|\mathcal{M}| \geq k$. Note that in the case
    where $|\mathcal{M}| \leq k$, the vertices in $C \setminus C'$ only have neighbors in $S \cup 
    C'$. 

    Let $S'$ be the union of the $C'$ for every such component $C$ where there exists a vertex 
    which belong to a triangle using one vertex from $s \in S$.
    Note that the number of those components $C$ is bounded by $|S|$. Indeed, $s$ cannot 
    be adjacent to any other component of $G -S$ or we have a paw using no edge from $S$ which is
    not possible. This implies that $|S'| \leq |S|( 4k+6)$.
\end{proof}

\section{Main Result}

We are now ready to prove our main theorem.
\begin{theorem}\label{th:editing}
    The paw-free editing problem has a kernel on $\bigO{k^6}$ vertices 
\end{theorem}

\begin{proof}
    Let $(G,k)$ be an instance of paw-free editing. The algorithm first apply Reduction Rule~\ref{rrule:twins} repetitively. 
    Once Reduction Rule~\ref{rrule:twins} cannot be applied anymore, the algorithm computes $\mathcal{H}$ a maximal 
    packing of edge-disjoint paws. If $\mathcal{H}$ consists of more than $k$ paws, answer \emph{no}. 
    If this is not the case, let $S$ be the set of vertices belonging to a paw of $\mathcal{H}$. 
    $|S| \leq 4k$. Then the algorithm apply Reduction Rules~\ref{rrule:triangles_sharing_edge}~
    and~\ref{rrule:large_part} until either $k < 0$, in which case it 
    answers \emph{no}, or they cannot be applied anymore. 
    
    Because $\mathcal{H}$ is maximal, Theorem \ref{th:structure} implies that the components $G-S$ 
    are either triangle-free or complete multipartite. Let $C$ be a complete multipartite component.
    If $|C| \geq (3k +3)(3k+5) $, then Lemma \ref{lem:mcc_size} implies that the algorithm 
    can apply Reduction Rule~\ref{rrule:large_complete_multipartite} or answer \emph{no}.
    Moreover Lemma \ref{lem:mcc_number} implies that the number of 
    complete multipartite components adjacent to $S$ is bounded by $|S|$. Overall this implies that the number of
    vertices contained in complete multipartite components of $G - S$ adjacent to $S$ is bounded by 
    $4k(3k +3)(3k+5)$, or it is possible to apply Reduction Rule~\ref{rrule:large_complete_multipartite}.

    By applying Lemma \ref{lemma:triangle_1}, we either find out that $(G,k)$ is a \emph{no}-instance 
    or find a set $S'$ of at most $( 4k+6)4k$ vertices such that if $x \not 
    \in S'$ belongs to a triangle-free component of $G-S$, then $x$ doesn't belong to any triangle 
    in $G$ using only one vertex of $S$. 

    Because Reduction Rule~\ref{rrule:triangles_sharing_edge} cannot be applied anymore, it means 
    that for every pair of adjacent vertices 
    $s_1, s_2$ in $S$, the number of vertices in triangle-free components adjacent to both 
    $s_1$ and $s_2$ is bounded by $4k +6$. This means that, if $S''$ denotes the set of vertices in 
    a triangle-free component forming a triangle with 2 vertices of $S$, then 
    $|S''| \leq |S|^2 (4k+6)$.

    Then we construct recursively sets $S_0, S_1, \dots S_6$ such that $S_i$ is a subset of 
    vertices of $G$ at distance $i$ from $S$ as follows: 
    First we set $S_0:= S$ and then, for every $i \in \{0,\ldots,5\}$, we define $S_{i+1}$ by picking, for 
    every vertex $x$ of $S_i$, $4k + 6$ neighbors of $x$ at distance $i+1$ from $S$ in $G$ and 
    belonging to a triangle-free component of $G-S$. Note that $|\bigcup S_i| = \bigO{k^6}$.

    Let $G'$ be the graph induced on $G$ by $S$, $S'$, $S''$ the $S_i$ and all the complete multipartite 
    components of $G-S$ adjancent to $S$. Note that, by construction of $S'$ and $S''$, there is no triangle in $G$ 
    using a vertex which is not in $G'$.
    We claim that $(G',k)$ has a solution if and only if $(G,k)$ has a solution. 
    As $G'$ is a subgraph of $G$, it is clear that if $(G,k)$
    has a solution, then so does $(G',k)$. Suppose now that $(G',k)$ has a solution $A$, but 
    $(G,k)$ does not have a solution. In particular, it implies that $G\Delta A$ is not paw-free.
    Because of Lemma \ref{lem:distance}, we can assume that no complete multipartite component of 
    $G' \Delta A$ has a vertex at distance $5$ from $S$ and that $A$ is minimal. 
    Let $x_1, x_2, x_3, x_4$ form a paw in $G \Delta A$, with $x_1, x_2, x_3$ being the triangle. 
    One of the $x_i$ must be a vertex which has not been marked during the construction of the 
    $S_i$. Moreover, since $G'$ contains all the triangles of $G$, it means that 
    $x_1, x_2$ and $x_3$ belong to $G'$ and $x_4$ doesn't. 
    It also means that $x_1, x_2$ and $x_3$ belong to a complete multipartite 
    component of $G' \Delta A$ and $x_4$ is adjacent to one of these vertices, say $x_1$. 
    Since $x_1$ is at distance less than $5$ from 
    $S$, it means that during the marking process $x_4$ was not marked for $x_1$. But this means
    that $x_1$ has more than $4k+6$ neighbors in triangle-free components of $G' - S$.
    However, Lemma \ref{lem:degre_trianglefree} implies that $x_1$ cannot belong 
    to a complete multipartite component of $G' \Delta A$, which is a contradiction. 
\end{proof}

\section{Better Bounds for Deletion and Addition}
 In this section, we provide better kernels for paw-deletion and 
 paw-addition. Let us start with the deletion problem, where the proof is quite similar to the one of 
 Theorem \ref{th:editing}, with the difference that we only keep vertices of the triangle-free components which are at distance one from $S$.  

 \begin{theorem}
    The paw-free deletion problem admits a kernel of size $\bigO{k^3}$.
 \end{theorem}

 \begin{proof}
    Let $(G,k)$ be an instance of paw-free deletion. First note that Reduction 
    Rules~\ref{rrule:twins}--\ref{rrule:large_part} are still 
    safe in this context, and Lemma \ref{lem:degre_trianglefree} still applies.
    Therefore the algorithm applies Reduction Rule~\ref{rrule:twins} until it cannot be applied 
    anymore. It then computes $\mathcal{H}$ a maximal packing of edge-disjoint paws. 
    If $\mathcal{H}$ consists of more than $k$ paws, answer \emph{no}. If this is not the case, let $S$ be 
    the set of vertices belonging to a paw of $\mathcal{H}$. $|S| \leq 4k$. Then the algorithm apply 
    Reduction Rules~\ref{rrule:triangles_sharing_edge}~and~\ref{rrule:large_part} until either $k < 0$, 
    in which case it answers \emph{no}, or they cannot be applied anymore. 

    Again, by possibly applying Reduction Rule~\ref{rrule:large_complete_multipartite}, we can assume
    that the set of vertices in all the multipartite components of $G - S$ adjacent to $S$ 
    is smaller than $4k(3k +3)(3k+5)$.
    By applying Lemma \ref{lemma:triangle_1}, we either find out that $(G,k)$ is a \emph{no}-instance 
    or find a set $S'$ of at most $( 4k+6)4k$ vertices such that if $x \not 
    \in S'$ belongs to a triangle-free component of $G-S$, then $x$ doesn't belong to any triangle 
    in $G$ using only one vertex of $S$. 

    Because Reduction Rule~\ref{rrule:triangles_sharing_edge} cannot be applied anymore, it means 
    that for every pair of adjacent vertices 
    $s_1, s_2$ in $S$, the number of vertices in triangle-free components adjacent to both 
    $s_1$ and $s_2$ is bounded by $4k +6$. This means that, if $S''$ denote the set of vertices in 
    a triangle-free component, forming a triangle with 2 vertices of $S$, then 
    $|S''| \leq |S|^2 (4k+6)$.

    Note also that Lemma \ref{lem:degre_trianglefree} still applies, and let $S_1$ be the set obtained 
    by picking for every vertex $s$ in $S$, $4k+6$ neighbors in triangle-free components of $G-S$.

    Let $G'$ be the graph induced on $G$ by $S$, $S'$, $S''$, $S_1$, as well as all the vertices on 
    complete multipartite components of $G-S$. 
    We want to show that $(G,k)$ has a solution if and only if $(G',k)$ has a solution. 
    Let $A$ be a solution of $(G',k)$ and suppose $G \Delta A$ has a paw $x_1, x_2, x_3, x_4$, 
    with $x_1, x_2, x_3$ being a triangle and $x_4$ being adjacent to $x_3$. Because of the choice 
    of the sets $S'$ and $S''$, all the triangle of $G$ are contained in $G'$. Note also that, since 
    the solution can only remove edges, $x_1, x_2, x_3$ is a triangle in $G$. This implies 
    that $x_3 \in S$ and $x_4$ was not picked for the $4k +6$ neighbors of $x_3$. In particular, this 
    means that $x_3$ has $4k+6$ neighbors which belong to a triangle-free component of $G'-S$ in 
    $G'$ and thus by Lemma \ref{lem:degre_trianglefree}, $x_3$ cannot belong to a complete multipartite 
    component of $G' \Delta A$. However, since $x_1, x_2$ and $x_3$ form a triangle in $G'\Delta A$, we reach 
    a contradiction.
 \end{proof}

\begin{theorem}
    The paw-free addition problem admits a kernel of size $\bigO{k^3}$.
\end{theorem}

\begin{proof}
    Again, Reduction Rules~\ref{rrule:twins}--\ref{rrule:large_part} are still safe in this context, 
    with the difference for Rules~\ref{rrule:triangles_sharing_edge} and \ref{rrule:large_part}
    that, instead of removing edges and decreasing $k$, we can directly conclude that $(G,k)$ is a \emph{no}-instance.
    Note also that a paw-free connected component can safely be removed from the graph. 
    
    So the algorithm start by removing all the paw-free components of $G$ and 
    applying Reduction Rule~\ref{rrule:twins} until it cannot be applied 
    anymore. It then computes $\mathcal{H}$ a maximal packing of edge-disjoint paws. 
    If $\mathcal{H}$ consists of more than $k$ paws, answer \emph{no}. If this is not the case, let $S$ be 
    the set of vertices belonging to a paw of $\mathcal{H}$. $|S| \leq 4k$. From now on we can 
    assume that Rules~\ref{rrule:triangles_sharing_edge}~and~\ref{rrule:large_part} cannot be applied.

    Again, by possibly applying Reduction Rule~\ref{rrule:large_complete_multipartite}, we can assume
    that the set of vertices in all the multipartite components of $G - S$ adjacent to $S$
    is smaller than $4k(3k +3)(3k+5)$.

    Consider a connected component $C_1$ of $G$. This component cannot be paw-free, or the algorithm
    would have removed it from the graph. So let $S_1 = C_1 \cap S$ and $R_1$ the vertices of $C_1$ contained in 
    triangle-free component of $G-S$. Because $C_1$ is not triangle-free, it means that any solution 
    $A$ to $(G,k)$ leaves $C_1$ as a complete multipartite component. In particular, it implies that 
    $R_1$ is smaller than $4k+6$. Indeed, if $R_1$ is bigger than $4k+6$, then $2k+6$ vertices 
    will have the same neighborhood in $G \Delta A$ as in $G$. Moreover, since $R_1$ is triangle-free, 
    it means that these vertices belong to at most $2$ parts of the complete multipartite component. 
    This implies that at least $k+3$ of these vertices belong to the same part and Rule~\ref{rrule:twins}
    applies. Moreover, since $G$ has at most $k$ connected component which are not paw-free, it implies 
    that the set of vertices contained in triangle-free components of $G-S$ is smaller than 
    $(4k+6)k$.

    Overall, it implies that our reduced instance has size at most $4k(3k +3)(3k+5) + (4k+6)k + 4k
    = \bigO{k^3}$, which ends the proof. 
\end{proof}

\section{Conclusion}
In this paper we studied \pawedit\ and gave a polynomial kernel of size $\bigO{k^{6}}$.  The only 
unresolved graphs $H$ on $4$ vertices, for which the kernelization complexity of \hedit\ problem remains open is claw. 
In fact, for this problem even the kernelization complexity of {\sc $H$-Edge Deletion} and {\sc $H$-Edge Addition} remain open. Settling the kernelization complexity might require using the power of structure theorem of claw free graphs. Thus,  a natural start here could be looking at editing/deletion/addition to basic graphs, on which structure theorem of claw free graphs is built. We leave these as natural directions to pursue.


\begin{thebibliography}{10}

    \bibitem{AravindSS17b}
    N.~R. Aravind, R.~B. Sandeep, and Naveen Sivadasan.
    \newblock Dichotomy results on the hardness of h-free edge modification
      problems.
    \newblock {\em {SIAM} J. Discrete Math.}, 31(1):542--561, 2017.
      {\path{doi:10.1137/16M1055797}}.
    
    \bibitem{bodlaender2006open}
    Hans~L Bodlaender, Leizhen Cai, Jianer Chen, Michael~R Fellows, Jan~Arne Telle,
      and D{\'a}niel Marx.
    \newblock Open problems in parameterized and exact computation-iwpec 2006.
    \newblock {\em UU-CS}, 2006, 2006.
    
    \bibitem{Cai96}
    Leizhen Cai.
    \newblock Fixed-parameter tractability of graph modification problems for
      hereditary properties.
    \newblock {\em Inf. Process. Lett.}, 58(4):171--176, 1996.
    \newblock \href {https://doi.org/10.1016/0020-0190(96)00050-6}
      {\path{doi:10.1016/0020-0190(96)00050-6}}.
    
    \bibitem{CaiCai15}
    Leizhen Cai and Yufei Cai.
    \newblock Incompressibility of h-free edge modification problems.
    \newblock {\em Algorithmica}, 71(3):731--757, 2015.
    \newblock \href {https://doi.org/10.1007/s00453-014-9937-x}
      {\path{doi:10.1007/s00453-014-9937-x}}.
    
    \bibitem{Caithesis}
    Yufei Cai.
    \newblock Polynomial kernelisation of {H}-free edge modification problems.
    \newblock {Mphil} thesis, Department of Computer Science and Engineering, The
      Chinese University of Hong Kong, Hong Kong SAR, China, 2012.
    
    \bibitem{CaoRSY18}
    Yixin Cao, Ashutosh Rai, R.~B. Sandeep, and Junjie Ye.
    \newblock A polynomial kernel for diamond-free editing.
    \newblock In {\em 26th Annual European Symposium on Algorithms, {ESA} 2018,
      August 20-22, 2018, Helsinki, Finland}, pages 10:1--10:13, 2018.
    \newblock \href {https://doi.org/10.4230/LIPIcs.ESA.2018.10}
      {\path{doi:10.4230/LIPIcs.ESA.2018.10}}.
    
    \bibitem{CyganFKLMPPS15}
    Marek Cygan, Fedor~V. Fomin, Lukasz Kowalik, Daniel Lokshtanov, D{\'{a}}niel
      Marx, Marcin Pilipczuk, Michal Pilipczuk, and Saket Saurabh.
    \newblock {\em Parameterized Algorithms}.
    \newblock Springer, 2015.
    
    \bibitem{diestel}
    R.~Diestel.
    \newblock {\em Graph Theory, 4th Edition}.
    \newblock Springer, 2012.
    
    \bibitem{DowneyFellows13}
    Rodney~G. Downey and Michael~R. Fellows.
    \newblock {\em Fundamentals of Parameterized Complexity}.
    \newblock Texts in Computer Science. Springer, 2013.
    
    \bibitem{ElmallahC88}
    Ehab~S. El-Mallah and Charles~J. Colbourn.
    \newblock The complexity of some edge deletion problems.
    \newblock {\em IEEE Transactions on Circuits and Systems}, 35(3):354--362,
      1988.
    \newblock \href {https://doi.org/10.1109/31.1748} {\path{doi:10.1109/31.1748}}.
    
    \bibitem{ErdosR60}
    Paul Erd{\"o}s and Richard Rado.
    \newblock Intersection theorems for systems of sets.
    \newblock {\em Journal of the London Mathematical Society}, 1(1):85--90, 1960.
    
    \bibitem{FlumG06}
    J{\"{o}}rg Flum and Martin Grohe.
    \newblock {\em Parameterized Complexity Theory}.
    \newblock Texts in Theoretical Computer Science. An {EATCS} Series. Springer,
      2006.
    \newblock \href {https://doi.org/10.1007/3-540-29953-X}
      {\path{doi:10.1007/3-540-29953-X}}.
    
    \bibitem{GuillemotHPP13}
    Sylvain Guillemot, Fr{\'{e}}d{\'{e}}ric Havet, Christophe Paul, and Anthony
      Perez.
    \newblock On the (non-)existence of polynomial kernels for \emph{P}
      \({}_{\mbox{ \emph{l} }}\)-free edge modification problems.
    \newblock {\em Algorithmica}, 65(4):900--926, 2013.
    \newblock \href {https://doi.org/10.1007/s00453-012-9619-5}
      {\path{doi:10.1007/s00453-012-9619-5}}.
    
    \bibitem{ImpagliazzoP01}
    Russell Impagliazzo and Ramamohan Paturi.
    \newblock On the complexity of k-sat.
    \newblock {\em J. Comput. Syst. Sci.}, 62(2):367--375, 2001.
    \newblock \href {https://doi.org/10.1006/jcss.2000.1727}
      {\path{doi:10.1006/jcss.2000.1727}}.
    
    \bibitem{KratschW13}
    Stefan Kratsch and Magnus Wahlstr{\"{o}}m.
    \newblock Two edge modification problems without polynomial kernels.
    \newblock {\em Discrete Optimization}, 10(3):193--199, 2013.
    \newblock \href {https://doi.org/10.1016/j.disopt.2013.02.001}
      {\path{doi:10.1016/j.disopt.2013.02.001}}.
    
    \bibitem{LewisY80}
    John~M. Lewis and Mihalis Yannakakis.
    \newblock The node-deletion problem for hereditary properties is np-complete.
    \newblock {\em J. Comput. Syst. Sci.}, 20(2):219--230, 1980.
    \newblock \href {https://doi.org/10.1016/0022-0000(80)90060-4}
      {\path{doi:10.1016/0022-0000(80)90060-4}}.
    
    \bibitem{OLARIU198853}
    Stephan Olariu.
    \newblock Paw-free graphs.
    \newblock {\em Information Processing Letters}, 28(1):53 -- 54, 1988.
    \newblock URL:
      \url{http://www.sciencedirect.com/science/article/pii/0020019088901433},
      \href {https://doi.org/https://doi.org/10.1016/0020-0190(88)90143-3}
      {\path{doi:https://doi.org/10.1016/0020-0190(88)90143-3}}.
    
    \bibitem{SandeepS15}
    R.~B. Sandeep and Naveen Sivadasan.
    \newblock Parameterized lower bound and improved kernel for diamond-free edge
      deletion.
    \newblock In {\em 10th International Symposium on Parameterized and Exact
      Computation, {IPEC} 2015, September 16-18, 2015, Patras, Greece}, pages
      365--376, 2015.
    \newblock \href {https://doi.org/10.4230/LIPIcs.IPEC.2015.365}
      {\path{doi:10.4230/LIPIcs.IPEC.2015.365}}.
    
    \bibitem{Yannakakis81}
    Mihalis Yannakakis.
    \newblock Edge-deletion problems.
    \newblock {\em {SIAM} J. Comput.}, 10(2):297--309, 1981.
\end{thebibliography}
\end{document}